\documentclass[10pt]{amsart}
\usepackage{amsmath,amssymb,amsthm,a4wide}
\usepackage[all]{xy}
\usepackage{graphicx}

\newtheorem{theorem}{Theorem}[section]    
\newtheorem{lemma}[theorem]{Lemma}          
\newtheorem{proposition}[theorem]{Proposition}  

\newtheorem{corollary}[theorem]{Corollary} 

\theoremstyle{definition}
\newtheorem{definition}[theorem]{Definition}

\newtheorem{example}[theorem]{Example}

 \makeatletter
    
    \@addtoreset{equation}{section}
  \makeatother
\def\co{\colon \thinspace}

\newcommand{\Z}{\mathbb{Z}}
\newcommand{\R}{\mathbb{R}}

\newcommand{\C}{\mathbb{C}}
\newcommand{\mB}{\mathcal{B}}
\newcommand{\mD}{\mathcal{D}}

\newcommand{\re}{\mathsf{Re}}
\newcommand{\im}{\mathsf{Im}}
\newcommand{\LWcr}{\textsf{LWcr}}
\newcommand{\SWcr}{\textsf{SWcr}}
\newcommand{\Wcr}{\textsf{Wcr}}
\newcommand{\LWin}{\textsf{LWin}}
\newcommand{\SWin}{\textsf{SWin}}
\newcommand{\Win}{\textsf{Win}}

\title[Curve diagrams for Artin groups of type B]{Curve diagrams for Artin groups of type B}
\author{Tetsuya Ito}

\address{Research Institute for Mathematical Sciences, Kyoto university
Kyoto, 606-8502, Japan}
\email{tetitoh@kurims.kyoto-u.ac.jp}
\urladdr{http://www.kurims.kyoto-u.ac.jp/~tetitoh/}

\begin{document}

\begin{abstract} 
We develop a theory of curve diagrams for Artin groups of type $B$. We define the winding number labeling and the wall crossing labeling of curve diagrams, and show that these labelings detect the classical and the dual Garside length, respectively. A remarkable point is that our argument does not require Garside theory machinery like normal forms, and is more geometric in nature.
\end{abstract}
\maketitle

\section{Introduction}

Let $\mB_{n}$ be the $n$-strand braid group defined by
\[ \mB_{n} = \left\langle \sigma_1,\ldots,\sigma_{n-1} \: 
\begin{array}{|cc}
\sigma_{i}\sigma_{j}\sigma_{i} = \sigma_{j}\sigma_{i}\sigma_{j}, & |i-j|=1 \\
\sigma_{i}\sigma_{j} = \sigma_{j}\sigma_{i}, & |i-j|>1
\end{array} \right\rangle. \]

$\mB_n$ is identified with the mapping class group of an $n$-punctured disc $D_{n}$, the group of diffeotopy classes of diffeomorphisms of $D_{n}$ that fix the boundary pointwise. Using this identification, one can represent a braid by a collection of smooth curves in $D_{n}$ called a {\em curve diagram} (See \cite[Chapter X]{ddrw}). Although the curve diagram representation is elementary, it reflects various deep properties of braids in a surprisingly simple way. For example, a curve diagram provides a geometric interpretation of the Dehornoy ordering of the braid groups \cite{fgrrw}, and one can read both the classical and the dual Garside lengths from the curve diagram \cite{iw,iw0,w} in a direct manner. Moreover, a certain simplifying procedure of curve diagrams provides a combinatorial model of the Teichm\"uller distance \cite{dw}. Thus, it is interesting to develop a theory of curve diagram for other groups that act on surfaces.

In the framework of the theory of Artin groups, the braid group $\mB_{n}$ is treated as an Artin group corresponding to the Dynkin diagram of type $A_{n-1}$. In this paper we deal with $A(B_{n})$, the Artin group corresponding to the Dynkin diagram of type $B_{n}$. The group $A(B_{n})$ is given by the presentation
\[ A(B_{n}) = \left\langle s_{1},\ldots,s_{n} \; 
\begin{array}{|cc}
 s_{1}s_{2}s_{1}s_{2}=s_{2}s_{1}s_{2}s_{1}, &  \\
 s_{i}s_{j} = s_{i}s_{j}, & |i-j|>1 \\
s_{i}s_{i+1}s_{i}=s_{i+1}s_{i}s_{i+1}, & i=2,\ldots,n-1
\end{array}
\right\rangle. \]

In this paper, we develop a theory of curve diagram for Artin groups of type B.
We introduce two labelings on curve diagrams, the {\em winding number labeling} and the {\em wall-crossing labeling} by generalizing the corresponding notions in the curve diagram of braids. 

In Theorem \ref{theorem:lengthformula_c} and Theorem \ref{theorem:lengthformula_dual}, we show that from these labelings one can read the {\em classical Garside length} and the {\em dual Garside length} of $A(B_{n})$. These are length functions of $A(B_{n})$ with respect to certain natural generating sets called the {\em classical simple elements} and the {\em dual simple elements}, respectively. Our main theorems provide a geometric and topological interpretation of such standard length functions.

The classical and the dual simple elements come from natural Garside structures on $A(B_{n})$. Here a Garside structure is a combinatorial and algebraic structure which produces an effectively computable normal form that solves the word and the conjugacy problems. An idea of Garside structure dates back to  Garside's solution of the word and the conjugacy problem of the braid groups \cite{ga}. See \cite{bgg,deh,dp} for the basics of theory of Garside groups.

A remarkable point in a curve diagram argument is that we require no deep Garside theory machinery. In particular, we do not need to use a lattice structure which is the key ingredient in Garside theory, so we do not use normal forms. Our requirement of algebraic properties for the classical and the dual simple elements, stated as Lemma \ref{lemma:keyproperty} and Lemma \ref{lemma:keyproperty_d} respectively, are much weaker than the requirement in developing Garside theory.
Thus, the result in this paper seems to suggest that one can construct a theory of curve diagrams for more general subgroups of the mapping class groups and there might be nice length functions, even if the group does not have a Garside structure.


\section{Curve diagram and its labelings for Artin groups of type B}

Let $D_{2n} =\{z \in \C \: | \: |z| \leq n+1 \} -\{-n,\ldots,-1,1,\ldots,n \}$ be the $2n$-punctured disc. For $i=1,\ldots,n$, we denote the puncture points $-i \in \C$ and $i \in \C$ by $p_{i}$ and $q_{i}$, respectively, and let $r: D_{2n} \rightarrow D_{2n}$ be the half-rotation of the disc $D_{2n}$ defined by $r(z) = -z$.

The braid group $\mB_{2n}= A(A_{2n-1})$ is identified with the mapping class group of $D_{2n}$ as follows: For $i=1,\ldots,n-1$ (resp. $i=n+1,\ldots,2n-1$), a standard generator $\sigma_{i}$ is identified with the isotopy class of the left-handed (clockwise) half Dehn twist along the segment of the real line $[p_{n-i+1},p_{n-i}]$ (resp. $[q_{i-n},q_{i-n+1}]$), and $\sigma_{n}$ is identified with the the isotopy class of the left-handed half Dehn twist along the segment $[p_{1},q_{1}]=[-1,1]$. See Figure \ref{fig:halfDehn}.

\begin{figure}[htbp]
\centerline{\includegraphics[width=70mm]{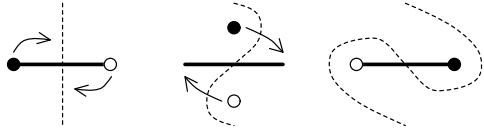}}
   \caption{Left-handed half Dehn twist along the line segment: The position of two punctures integerchanges along the line segment.}
 \label{fig:halfDehn}
\end{figure}

To define the curve diagrams for $A(B_{n})$, we consider the homomorphism 
\[ \Psi:A(B_{n}) \rightarrow \mB_{2n} \]
defined by 
\begin{gather}
\label{eqn:psi}
 \Psi(s_{i}) = 
\begin{cases}
\sigma_{n} & (i=1), \\
\sigma_{n+(i-1)}\sigma_{n-(i+1)} & (i>1).
\end{cases}
\end{gather} 

It is well-known that $\Psi$ is injective.
To see this geometrically, it is convenient to first identify $A(B_{n})$ as the subgroup of the mapping class group of $D_{n+1}$ that preserves the first puncture point $p_{0}$. We regard $D_{2n}$ as the double branched covering of $D_{n+1}$ branched at $p_{0}$. Then $\Psi$ is the map obtained by taking the lift, and is known to be injective by famous Birman-Hilden theorem \cite{bh}.

 Using $\Psi$, we regard an element of $A(B_n)$ as an element the mapping class group of $D_{2n}$. 

Let $\overline{E}$ be the diagram in $D_{2n}$ consisting of the real line segment between the point $-(n+1)$, the leftmost point of $\partial D_{2n}$, and $q_{1}$. Similarly, let $E$ be the diagram in $D_{2n}$ consisting of the real line segment between $p_{n}$ and $q_{1}$. Both $\overline{E}$ and $E$ are oriented from left to right. We denote the line segment of $\overline{E}$ connecting $-(n+1)$ and $p_{n}$ by $E_{0}$, the line segment connecting $p_{n-i+1}$ and $p_{n-i}$ by $E_{i}$ ($i=1,\ldots,n-1$), and the line segment connecting $p_{1}$ and $q_{1}$ by $E_{n}$.

For $i=1,\ldots,n$, let $W_{i}$ be a vertical line segment in the upper half-disc $\{z \in D^{2} \: | \: \textrm{Im}\, z>0 \: \}$ oriented upwards which connects the puncture $p_{i}$ and a point in $\partial D_{2n}$. Similarly, let $W_{i+n}$ be a vertical line segment in the lower half-disc $\{z \in D^{2} \: | \: \textrm{Im}\, z<0 \: \}$ oriented downwards which connects the punture $q_{i}$ and a point in $\partial D_{n}$. See Figure~\ref{fig:curvediagram} (a). We call $W_i$ the {\em walls}, and their union $\bigcup W_{i}$ is denoted $W$. Observe that $r(W_{i})=W_{i+n}$.

\begin{definition}[Curve diagram]
\label{defn:curvediagram}
For $\beta \in A(B_{n})$, the {\em total curve diagram} and the {\em curve diagram} of $\beta$ is the image of the diagrams $\overline{E}$ and $E$, respectively, under a diffeomorphism $\phi$ representing $\Psi(\beta)$ which satisfies the following conditions.

\begin{enumerate}
\item[(i)] $\phi(\overline{E})$ is transverse to $W$, and the number of intersections of $\phi(\overline{E})$ with $W$ is minimal in its diffeotopy class.
\item[(ii)] The number of vertical tangencies (the points $p$ of $\phi(\overline{E})$ where the tangent vector at $p$ is vertical) is minimal in its diffeotopy class.
\item[(iii)] For each puncture point $z \in \{-n,\ldots,-1,1,\ldots,n\}$, there exists a small disc neighborhood $B(z)$ of $z$ such that $B(z) \cap \phi(\overline{E})$ coincides with the real line.
\item[(iv)] $r(\phi(E_{n}))=\phi(E_{n})$. 
\end{enumerate}
\end{definition}

See Figure \ref{fig:curvediagram} (b) for an example. We will use a dotted line to represent $\phi(E_0)$. 
We denote the curve diagram of $\beta$ by $D_{\beta}$ and the total curve diagram by $\overline{D_\beta}$, respectively. Up to diffeotopy, a curve diagram is uniquely determined by $\beta$, so from now on we will often identify an element $\beta \in A(B_{n})$ with its representative diffeomorphism $\phi$ that produces the curve diagram of $\beta$.

Although to develop a theory of curve diagram it is sufficient to consider $D_{\beta}$ and $\overline{D_\beta}$, it is often convenient to make curve diagrams $r$-symmetric by considering $\mD_{\beta} = D_{\beta} \cup r(D_{\beta})$ and $\overline{\mD_\beta} = \overline{D_\beta} \cup r(\overline{D_\beta})$. We call $\mD_{\beta}$ (resp. $\overline{\mD_\beta}$) the {\em completed curve diagram} (resp. the {\em completed total curve diagram}). 

\begin{figure}[htbp]
\centerline{\includegraphics[width=110mm]{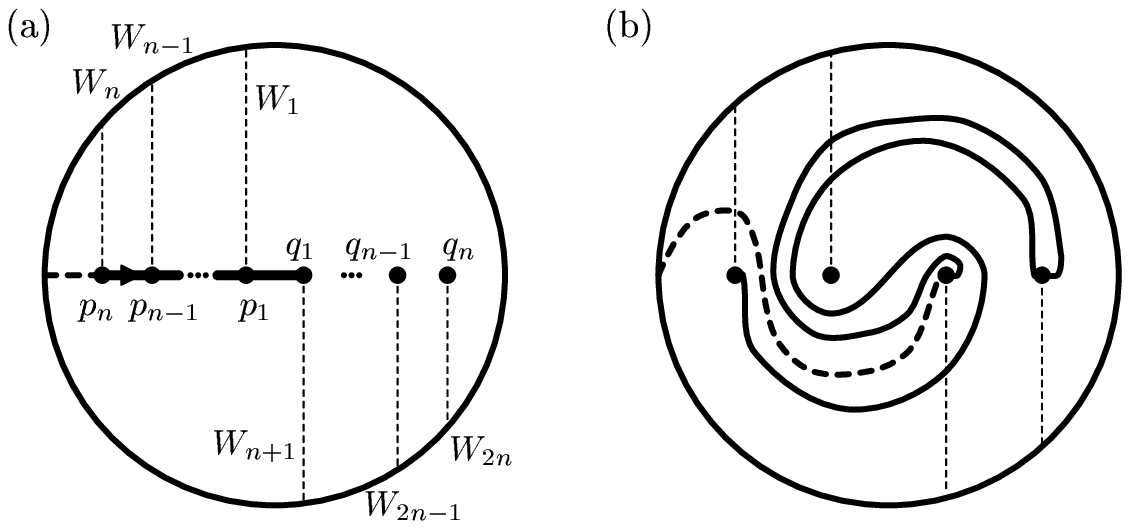}}
   \caption{(a) Punctured disc $D_{2n}$, walls and diagram $\overline{E}$, (b) Curve diagram of $s_{1}^{-1}s_{2}s_{1} \in A(B_{2})$.}
 \label{fig:curvediagram}
\end{figure}

We denote the union of the neighborhood $B(z)$ in Definition \ref{defn:curvediagram} (iii) by $B$. A point $x$ on a curve diagram which is not contained in $B$ is called {\em regular} if $x$ is neither a vertical tangency nor an intersection point with walls. For a regular point of the curve diagram, we assign two integers, the {\em winding number labeling} and the {\em wall-crossing labeling} as follows.

To introduce labelings, we temporary modify the curve diagram near the puncture points. For each puncture point $z$ that lies on $D_{\beta}$ other than $\beta(q_{1})$, we modify the curve diagram $D_{\beta}$ in $B(z)$ as shown in Figure \ref{fig:modification}, to miss the punctures.
Then the resulting diagram can be regarded as an arc in $D_{2n}$, which we still call the curve diagram of $\beta$ by abuse of notation.

\begin{figure}[htbp]
\centerline{\includegraphics[width=55mm]{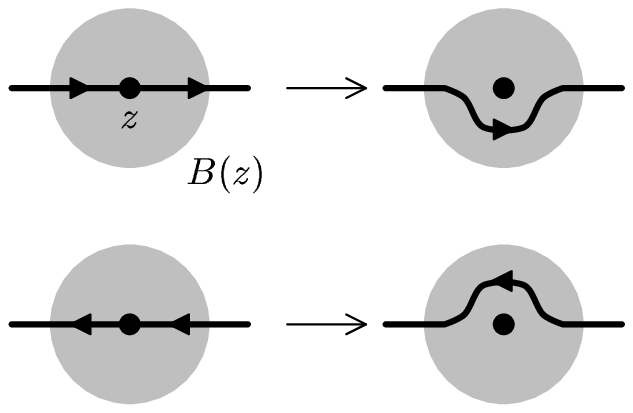}}
   \caption{Modification near puncture points}
 \label{fig:modification}
\end{figure}

Take a smooth parametrization of the modified version of a curve diagram $\gamma\co [0,1] \rightarrow D^{2}\subset \C$ and let $T_{\gamma} : [0,1] \rightarrow S^{1}=\R\slash \Z$ be the direction map defined by $T_{\gamma}(t) = \gamma'(t)/\|\gamma'(t)\|$.  
Take a lift $\widetilde{T_{\gamma}}: [0,1] \rightarrow \R$ of $T_{\gamma}$ so that $\widetilde{T_{\gamma}}(0) = 0$. Then $\widetilde{T_{\gamma}}(t) \in \Z +\frac{1}{2}$ if and only if $\gamma(t)$ is a vertical tangency. 
For a regular point $x=\gamma(t) \in \overline{D_\beta}$, we assign the integer $\Win(x)=R(\widetilde{T_{\gamma}}(t))$, where $R: \R \rightarrow \Z$ is a rounding function which sends real numbers to the nearest integers. We call $\Win(x)$ the {\em winding number labeling} at $x$.
Similarly, we assign the integer $\Wcr(x)$ defined by the algebraic intersection number of the arc $\gamma([0,t])$ and walls $W$. We call $\Wcr(x)$ the {\em wall crossing labeling} at $x$.

Geometrically, these definitions say that the winding number labeling counts how many times the curve $\gamma([0,t])$ winds the plane and the wall-crossing labeling counts how many times $\gamma([0,t])$ crosses the walls.

\begin{example}
Figure \ref{fig:label_example} shows an example of the winding and the wall crossing labeling for $\beta=s_{1}^{-1}s_{2}s_{1}\in A(B_{2})$. The classical and the dual normal forms of $\beta$ are $N_{\sf classical}(\beta) = (s_{2}s_{1}s_{2})(s_{2}s_{1})\Delta^{-1}$ and $N_{\sf dual}(\beta) = (s_{1}^{-1}s_{2}s_{1})$, respectively.

\begin{figure}[htbp]
\centerline{\includegraphics[width=110mm]{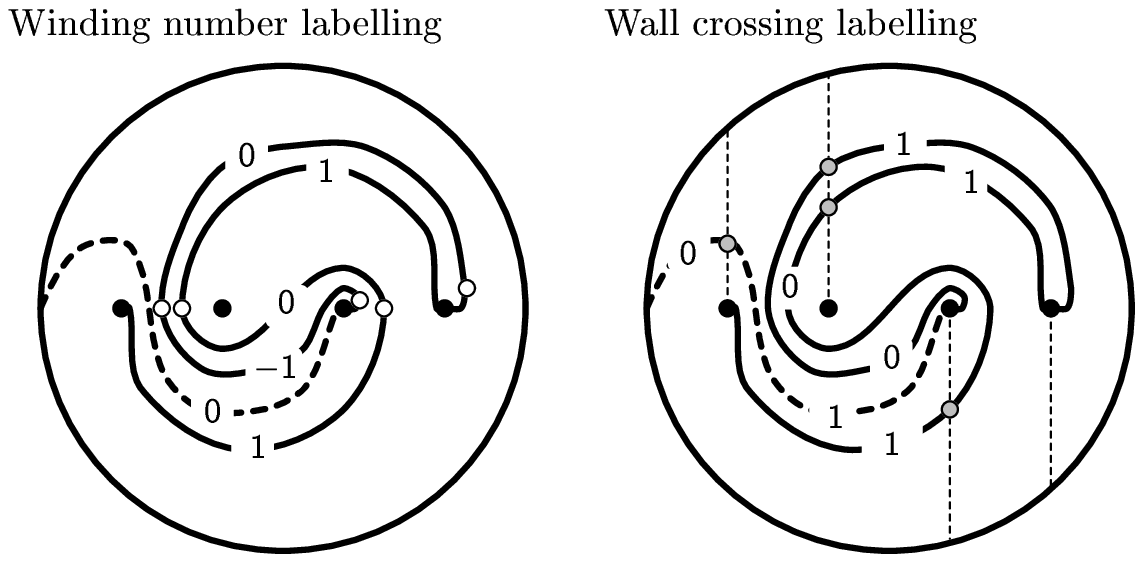}}
   \caption{The winding number and the wall crossing labelings for the curve diagram of $\beta=s_{1}^{-1}s_{2}s_{1} \in A(B_{2})$. In the left figure, white circles represent vertical tangencies, and in the right figure, gray circles represent the intersections with walls.}
 \label{fig:label_example}
\end{figure}
\end{example}

The winding number and the wall crossing labelings for the completed (total) curve diagrams $\mD_{\beta}$ ($\overline{\mD_\beta}$) are defined in a similar way. Since the action of $A(B_{n})$ on $D_{2n}$ and all the ingredients appearing in the definition of labelings, such as winding numbers or walls, are $r$-symmetric, the labelings of $r(D_{\beta})$ is determined by $D_{\beta}$.
That is, for a regular point $x \in D_{\beta}$, we have an $r$-symmetry
\begin{equation}
\label{eqn:symmetry}
 \Win(r(x))=\Win(x) \textrm{ and } \Wcr(r(x))=\Wcr(x).
\end{equation}

By Definition \ref{defn:curvediagram} (iv), $r(\beta(E_n))$ is, as a curve, identical with $\beta(E_n)$ but their orientations are opposite. However, (\ref{eqn:symmetry}) says that the labelings of the completed curve diagram $\mD_{\beta}$ is well-defined on $r(\beta(E_n)) = \beta(E_{n})$.

For $\beta \in A(B_n)$, we define $\LWin(\beta)$ and $\LWcr(\beta)$ as the largest winding number and wall crossing number labelings occurring in $D_{\beta}$. Similarly, we define $\SWin(\beta)$ and $\SWcr(\beta)$ as the smallest winding number and wall crossing number labelings in $D_{\beta}$. 
We remark that to define these numbers, we only consider the labelings of the curve diagram $D_{\beta}$, not the total curve digram $\overline{D_{\beta}}$. However, we need the total curve diagram in order to define the labelings.

The following is a direct consequence of the definition of labelings.
 
\begin{lemma}
For $\beta \in A(B_{n})$, the following three conditions are equivalent.
\begin{enumerate}
\item $\beta=1$.
\item $\SWin(\beta)=\LWin(\beta)=0$.
\item $\SWcr(\beta)=\LWcr(\beta)=0$.
\end{enumerate}
\end{lemma}

\section{Length formula}

\subsection{Classical Garside length and winding number labelings} 

To state our main theorem, first we recall the definitions of the classical simple elements. Since we want to avoid algebraic machinery as possible, we use the following geometric definition.

\begin{definition}
\label{defn:csimple}
An element $x \in A(B_{n})$ is called a classical simple element if as a mapping class, $x$ is described as the following $r$-symmetric dance of punctures:
\begin{description}
\item[Step 1] Perform a clockwise rotation of angle $\pi\slash 2$ so that all punctures $\{p_1,\ldots,p_n,q_1,\ldots,q_n\}$ lie on the imaginary axis.
\item[Step 2] Move punctures horizontally so that the followings are satisfied:
\begin{enumerate}
\item $\{\re(p_1),\ldots,\re(p_n),\re(q_{1}), \ldots, \re(q_{n})\} =\{-n,\ldots, -1,1,\ldots,n\}$.  
\item $r(p_{i})=q_{i}$ holds for all $i=1, \ldots,n$.
\end{enumerate}
\item[Step 3]
Move the punctures vertically so that all punctures lie on the real axis.
\end{description}

The classical Garside element $\Delta$ is an element of $A(B_{n})$ that corresponds to the clockwise half-rotation of the disc $D_{2n}$. 
\end{definition}

We denote the set of all classical simple elements by $[1,\Delta]$. 
Since the standard generators $s_i$ are classical simple elements, $[1,\Delta]$ generates $A(B_{n})$. For $\beta \in A(B_n)$, the {\em classical Garside length} $\ell_{\sf classical}(\beta)$ is the length of $\beta$ with respect to the classical simple elements $[1,\Delta]$.

Now we are ready to state the first main theorem of this paper, which generalizes the corresponding theorem for curve diagrams of braid groups \cite[Theorem 2.1]{w}.

\begin{theorem}
\label{theorem:lengthformula_c}
For $\beta \in A(B_{n})$, $\ell_{\sf classical}(\beta) = \max\{ \LWin(\beta),0 \} - \min \{ \SWin(\beta),0 \}$.
\end{theorem}

It is well-known that classical Garside elements and the classical simple elements have various nice algebraic properties \cite{bs,del}. However, to prove our main theorem and develop a curve diagram theory, we only need the  following, which is directly confirmed from the definition. 

\begin{lemma}
\label{lemma:keyproperty}
If $x \in [1,\Delta]$, then both $\Delta x^{-1}$ and $x^{-1}\Delta$ lies in $[1,\Delta]$.
\end{lemma}

To prove theorem \ref{theorem:lengthformula_c}, it is sufficient to observe the following. Recall that $\beta \in A(B_{n})$ is {\em classical positive} (resp. {\em classical negative}) if $\beta$ is written as a product of positive (resp. negative) classical simple elements $[1,\Delta]$.

\begin{proposition}
\label{proposition:key}
If $\beta \in A(B_{n})$ is classical positive, then $\LWin(\beta) = \ell_{\sf classical}(\beta)$ and $\SWin(\beta)\geq 0$.
Similarly, if $\beta \in A(B_{n})$ is classical negative, then $\SWin(\beta) = \ell_{\sf classical}(\beta)$ and $\LWcr(\beta) \leq 0$.
\end{proposition}

\begin{proof}[Proof of Theorem \ref{theorem:lengthformula_c}, assuming Proposition \ref{proposition:key}]

For $\beta \in A(B_{n})$, let us take a geodesic representative of $\beta$ with respect to $\ell_{\sf classical}$, $\beta=x_{1}^{\varepsilon_{1}}\cdots x_{\ell}^{\varepsilon_\ell}$ $(x_{i} \in [1,\Delta], \varepsilon_{i} \in \{\pm 1\})$.
Let $\ell_{\sf p}$ and $\ell_{\sf n}$ be the number of $i$ such that $\varepsilon_{i}=+1$ and $\varepsilon_{i}=-1$, respectively. If either $\ell_{\sf p}$ or $\ell_{\sf n}$ is zero, then $\beta$ is either classical positive or classical negative so we are done by Proposition \ref{proposition:key}.
Thus we assume neither $\ell_{\sf p}$ nor $\ell_{\sf n}$ is zero.

By Lemma \ref{lemma:keyproperty}, we may rewrite the geodesic word as 
$\beta = z_{\ell}z_{\ell-1}\cdots z_{1}\Delta^{-\ell_{\sf n}}$, $z_{i} \in [1,\Delta]$. As an element of mapping class group, $\Delta$ is a half rotation of the disc $D_{2n}$, hence $\LWin(\beta\Delta^{\ell_{\sf n}}) = \LWin(\beta)+\ell_{\sf n}$. On the other hand, the braid $\beta\Delta^{\ell_{\sf n}}=z_{\ell}z_{\ell-1}\cdots z_{1}$ is classical positive, hence by Proposition \ref{proposition:key}, $\LWin(\beta\Delta^{\ell_{\sf n}})=\ell_{\sf classical}(z_{\ell}z_{\ell-1}\cdots z_{1}) \leq \ell$. If $\ell' = \ell_{\sf classical}(z_{\ell}z_{\ell-1}\cdots z_{1}) < \ell$, then we may write $\beta = z'_{\ell'}\cdots z'_{1}\Delta^{-\ell_{\sf n}}$. By using Lemma  \ref{lemma:keyproperty}, we obtain a shorter word representative of $\beta$ which is impossible since $\ell=\ell_{\sf classical}(\beta)$. This shows $\LWin(\beta\Delta^{\ell_{\sf n}})=\ell$ hence 
\begin{equation} 
\label{eqn:LWin}
\LWin(\beta) = \ell - \ell_{\sf n}. 
\end{equation}

Similarly, by Lemma \ref{lemma:keyproperty}, we may rewrite the geodesic word as $\beta = z'_{\ell}{}^{-1}\cdots z'_{1}{}^{-1}\Delta^{-\ell_{\sf n} +\ell}$, and $\SWin(\beta\Delta^{\ell_{\sf n} - \ell}) = \SWin(\beta)+\ell_{\sf n}-\ell$ holds.
By Proposition \ref{proposition:key}, $\SWin(\beta\Delta^{\ell_{\sf n} -\ell})=-\ell_{\sf classical}(z'_{\ell}{}^{-1}\cdots z'_{1}{}^{-1}) = -\ell$ hence we conclude 
\begin{equation} 
\label{eqn:SWin}
\SWin(\beta) = -\ell_{\sf n}. 
\end{equation}
(\ref{eqn:LWin}) and (\ref{eqn:SWin}) prove Theorem \ref{theorem:lengthformula_c}.
\end{proof}

It remains to show Proposition \ref{proposition:key}.
The following lemma shows that we have an effective untangling procedure of the curve diagram.

\begin{lemma}
\label{lemma:win}
Let $\beta \in A(B_{n})$ be a non-trivial element such that $\SWin(\beta) \geq 0$. Then there exists a classical simple element $x$ such that 
\begin{enumerate}
\item $\LWin(x^{-1}\beta)=\LWin(\beta)-1$.
\item $\SWin(x^{-1}\beta) \geq 0$. 
\end{enumerate}
\end{lemma}
\begin{proof}
First we express an action of the inverse of classical simple elements as a three-step move of punctures that is a converse of the action given in Definition \ref{defn:csimple}.
\begin{description}
\item[Step 1] Move punctures vertically so that 
\[ \{\im(p_1),\ldots, \im(p_n), \im(q_{1}), \ldots, \im(q_{n})\} =\{-n,\ldots, -1,1,\ldots,n\} \]
and $r(p_{i})=q_{i}$ for all $i=1,\ldots,n$
\item[Step 2] Move punctures horizontally so that all punctures lie on the imaginary axis.
\item[Step 3] Perform an counter-clockwise rotation of angle $\pi\slash 2$ so that all puncture points lie on the real axis.
\end{description}

Here in the {\bf Step 1}, we need to determine the imaginary part of punctures. 
From the (completed) curve diagram $\mD_{\beta}$ we define the partial ordering $\prec$ on the set of puncture points in the following manner.

Consider the connected components of $\mD_{\beta}- \{\textrm{vertical tangencies}\}$. We will call such arcs {\em $V$-arcs}. Each $V$-arc may contain more than one puncture points, and the winding number labelings take a constant value on each $V$-arc. Roughly speaking, we define $z \prec z'$ if there exists a $V$-arc $\alpha$ such that $z'$ lies above $\alpha$ and $z$ lies below $\alpha$. 

To define $\prec$ precisely, observe that there are two types of $V$-arc $\alpha$: The first case is that the winding number labeling $\Win$ takes a local maxima or local minima on $\alpha$, in other words, at the endpoints of $\alpha$ the direction of winding is different. For such $V$-arc $\alpha$, we move punctures vertically and isotope the diagram accordingly so that the resulting $V$-arc does not contain horizontal tangencies (See Figure \ref{fig:movep} (A)).

The second case is that the winding number labeling $\Win$ does not take a local maxima or local minima on $\alpha$, equivalently saying, at the endpoints of $\alpha$ the direction of winding is the same. For such $V$-arc $\alpha$, we move punctures vertically and isotope the diagram accordingly so that the resulting $V$-arc is horizontal except near vertical tangency. (See Figure \ref{fig:movep} (B)).

After these moves, by comparing the imaginary part we get a partial ordering $\prec$ (c.f. \cite[Sublemma 2.3]{w}). Since $\mD_{\beta}$ is $r$-symmetric, we can perform the move of punctures so that it is $r$-symmetric. In particular, the resulting partial ordering $\prec$ can be chosen so that it is $r$-antisymmetric: $z \prec z'$ implies $r(z) \succ r(z')$. Let $\widetilde{\prec}$ be an $r$-antisymmetric total ordering on the the set of punctures that extends $\prec$. Then $\widetilde{\prec}$ determines the imaginary part of the punctures in {\bf Step 1}. The $r$-antisymmetry of $\widetilde{\prec}$ implies that the move of punctures described in {\bf Step 1} is $r$-symmetric in the sense $r(p_{i})=q_{i}$.

The moves in {\bf Steps 1--3} defines the inverse of a classical simple element $x$. From the definition of $\prec$, the vertical moves of punctures in {\bf Step 1} removes the $V$-arcs with labeling $\LWin(\beta)$. Hence $\LWin(\beta)$ decreases by one after performing $x^{-1}$, so $\LWin(x^{-1}\beta)=\LWin(\beta)-1$.
Similarly, the vertical moves of punctures in {\bf Step 1} does not affect the labelling of the $V$-arcs with labeling $\SWin(\beta)$. (See \cite{w} for more detailed explanation)

This shows $\SWin(x^{-1}\beta)=\SWin(\beta)\geq 0$. The case $x=\Delta$ happens only if $\LWin(\beta)=\SWin(\beta)$, so in this case $\SWin(x^{-1}\beta)\geq 0$ is also satisfied. 

\begin{figure}[htbp]
\centerline{\includegraphics[width=90mm]{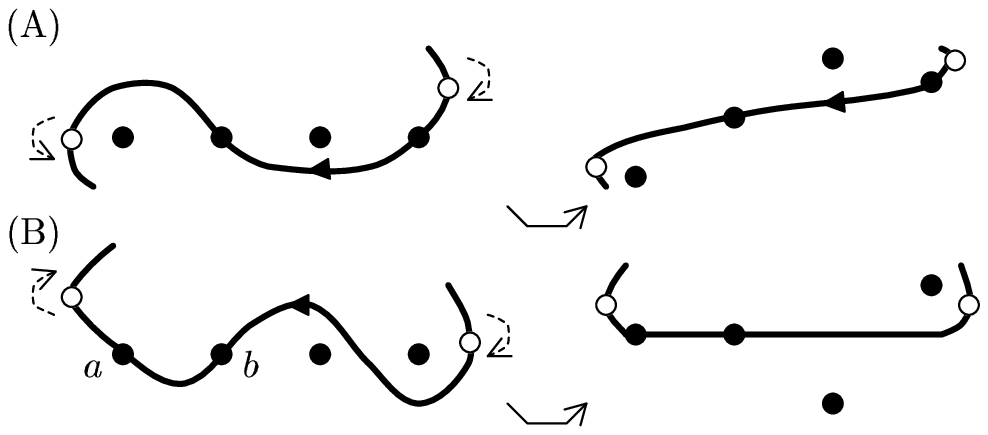}}
   \caption{How to determine the imaginary part of punctures. (A) illustrates the case that at the endpoints (the vertical tangencies, denoted by white circles), the direction of winding (indicated by dotted arrow) disagrees. (B) illustrates the case that the direction of winding agrees. In such case, between the puncture points $a$ and $b$ lying on the $V$-arc the partial ordering $\prec$ is not defined.}
\label{fig:movep}
\end{figure}
\end{proof}

\begin{proof}[Proof of Proposition \ref{proposition:key}]

The action of classical simple elements of $A(B_{n})$, as given in Definition \ref{defn:csimple}, shows that a classical simple element $x$ acts on $D_{n}$ locally as clockwise rotations so never decreases the winding number labelings. 
Hence $\SWin(\beta) \geq 0$ for a classical positive braid $\beta$.
Moreover, $x$ add windings to each $V$-arcs at most by onem hence $\LWin(x \beta) \leq \LWin(\beta)+1$ and $\LWin(x^{-1}\beta) \leq \LWin(\beta)$ for all $\beta \in A(B_{n})$ and $x \in [1,\Delta]$.
(see \cite{w} for more detailed explanation). 
In particular, we have an inequality $ \LWin(\beta) \leq \ell_{\sf classical}(\beta)$
for any (not necessarily classical positive) $\beta \in A(B_{n})$.

If $\beta$ is classical positive then $\SWin(\beta)\geq 0$, so Lemma \ref{lemma:win} shows a classical positive $\beta \in A(B_{n})$ can be written as a product of $\LWin(\beta)$ classical positive elements. So we get the converse inequality $\ell_{\sf classical}(\beta) \leq \LWin(\beta)$.
So we conclude $\ell_{\sf classical}(\beta) = \LWin(\beta)$.
The assertion for $\SWin(\beta)$ is proved in a similar manner.
\end{proof}

\subsection{Dual Garside length and wall-crossing labeling}

In a similar manner, we prove the length formula for the dual Garside length.
To define dual simple elements, we isotope the punctures, walls, and curves $\overline{E}$ so that all punctures lie on the circle $|z|=n$, preserving the property that $r(W_{i}) = W_{n+i}$ (see the left hand side of the Figure \ref{fig:curveb_dual}). Since the wall-crossing labeling is defined in terms of the algebraic intersection numbers, this isotopy does not affect the wall crossing labeling.

\begin{figure}[htbp]
\centerline{\includegraphics[width=110mm]{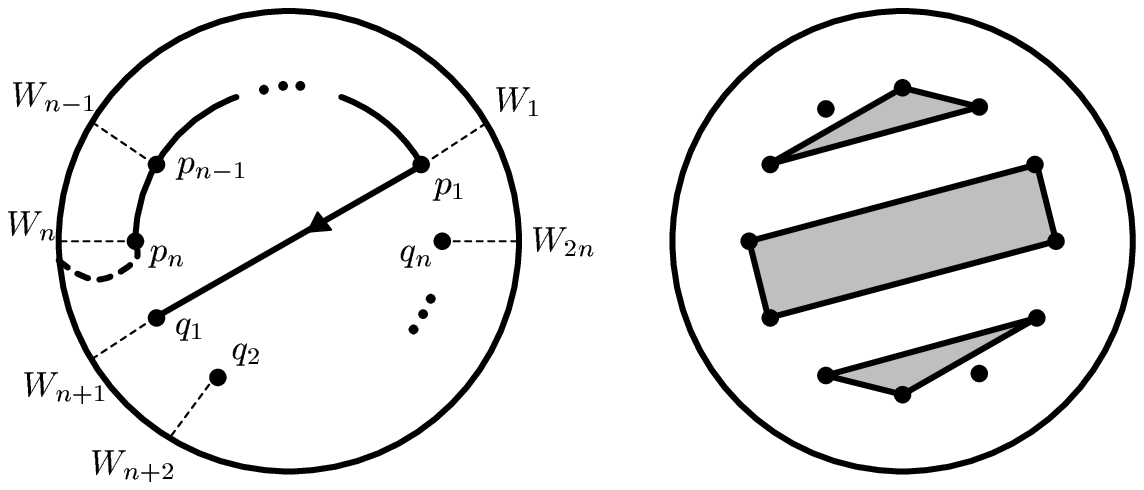}}
   \caption{(Left) Isotoping walls and curve diagram to treat dual simple elements. (Right) $r$-symmetric convex polygons.}
 \label{fig:curveb_dual}
\end{figure}

Take a collection of convex polygons $Q$ in $D_{2n}$ whose vertices are puncture points. We say $Q$ is {\em $r$-symmetric} if $r(Q)=Q$ (see the right of the Figure \ref{fig:curveb_dual}, for example).

For an $r$-symmetric collection of convex polygons $Q$, we define $y_{Q}\in A(B_{n})$ as follows. For each connected component $Q'$ of $Q$, we associate a move of puncture points that corresponds to the clockwise rotation of $Q'$. Namely, each puncture on $Q'$ moves to the adjacent punctures of $Q'$ in the clockwise direction along the boundary of $Q$ (see Figure \ref{fig:dualsimple}).
If $Q'$ is degenerate, namely, $Q'$ is a line segment $e$ connecting two punctures, the resulting move is nothing but the half Dehn twist along $e$ which we described in Figure \ref{fig:halfDehn}. 
This move of puntures defines an element $y_{Q'} \in \mathcal{B}_{2n}$.
We define
\[ y_{Q}= \prod_{Q'} y_{Q'} \]
where $Q'$ runs all connected components of $Q$.
Since $Q$ is $r$-symmetric, $y_{Q} \in A(B_{n})$.

\begin{figure}[htbp]
\centerline{\includegraphics[width=90mm]{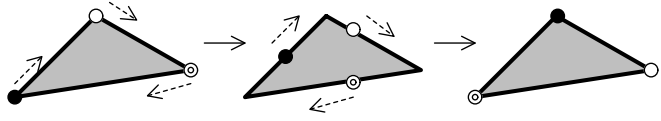}}
   \caption{The action of dual simple elements}
 \label{fig:dualsimple}
\end{figure}

\begin{definition}
\label{defn:dsimple}
An element $y \in A(B_{n})$ is called a \emph{dual simple element} if $y=y_{Q}$ for some $r$-symmetric collection of convex polygons $Q$. 
The dual Garside element $\delta$ is a dual simple element that corresponds to the connected convex polygon $Q$ having all the punctures as its vertices. 
We denote the set of dual simple elements by $[1,\delta]$. 
\end{definition}

As an element of mapping class group, $\delta$ is nothing but the rotation of $D_{2n}$ by $\frac{\pi}{n}$. Since the standard generators $s_{i}$ are dual simple elements, $[1,\delta]$ generates $A(B_{n})$. For $\beta \in A(B_n)$, the {\em dual Garside length} $\ell_{\sf dual}(\beta)$ is the length of $\beta$ with respect to the dual simple elements $[1,\delta]$.

Now we are ready to state the second main theorem, which generalizes the corresponding theorem for curve diagrams of the braid groups \cite{iw}.

\begin{theorem}
\label{theorem:lengthformula_dual}
For $\beta \in A(B_{n})$, $\ell_{\sf dual}(\beta) = \max\{ \LWcr(\beta),0\} - \min \{ \SWcr(\beta),0 \}$.
\end{theorem}

The proof of Theorem \ref{theorem:lengthformula_dual} is similar to the proof of Theorem \ref{theorem:lengthformula_c}.

First observe that from the definition, the dual simple elements also have the same property as the classical simple elements.

\begin{lemma}
\label{lemma:keyproperty_d}
If $y \in [1,\delta]$, then both $\delta y^{-1}$ and $y^{-1}\delta$ lies in $[1,\delta]$.
\end{lemma}

Recall that we say $\beta \in A(B_{n})$ is \emph{dual positive} (resp. \emph{dual negative}) if $\beta$ is written as a product of positive (resp. negative) dual simple elements. 
By the same argument as the proof of Theorem \ref{theorem:lengthformula_c},
the following Proposition \ref{proposition:key_dual} and Lemma \ref{lemma:keyproperty_d} proves Theorem \ref{theorem:lengthformula_dual}.

\begin{proposition}
\label{proposition:key_dual}
If $\beta \in A(B_{n})$ is dual positive, then $\LWcr(\beta) = \ell_{\sf dual}(\beta)$ and $\SWcr(\beta) \geq 0$.
Similarly, if $\beta \in A(B_{n})$ is dual negative, then $\SWcr(\beta) = -\ell_{\sf dual}(\beta)$ and $\LWcr(\beta) \leq 0$.
\end{proposition}

\begin{proof}[Proof of Proposition \ref{proposition:key_dual}]
First we show the wall-crossing labeling counterpart of Lemma \ref{lemma:win}:
If $\SWcr(\beta) \geq 0$, then there exists a dual simple element $y$ such that $\LWcr(y^{-1}\beta) = \LWcr(\beta)-1$ and that $\SWcr(y^{-1}\beta)\geq 0$. This proves $\LWcr(\beta) \leq \ell_{\sf dual}(\beta)$ if $\beta$ is dual positive.

Let $\mathcal{A}$ be the set of arcs in $D_{\beta}-W$ that attain the largest value of the wall-crossing labelings. Each arc $a \in \mathcal{A}$ connects two distinct walls, say $i(a)$-th and $j(a)$-th wall. For $a \in \mathcal{A}$, we denote the straight line in $D_{2n}$ connecting two punctures $p_{i(a)}$ and $p_{j(a)}$ by $e(a)$.
Let $Q$ be the convex hull of $\bigcup_{a\in \mathcal{A}} e(a)$ in $D_{2n}$. Since the curve diagram is $r$-symmetric, so is $Q$. Hence $Q$ defines a dual simple element $y$ of $A(B_n)$. By definition of $Q$, multiplying by $y^{-1}$ removes arcs with wall-crossing labeling $\LWcr(\beta)$ preserving $\SWcr(\beta) \geq 0$, as desired.
See \cite{ci,iw} for more detailed discussion.

To get the converse inequality, recall that the action of a dual simple element is by rotations of convex polygons. Thus, $\LWcr(y\beta) \leq \LWcr(\beta) +1$ and $ \LWcr(y^{-1} \beta) \leq \LWcr(\beta) +1 $ hold for any $\beta \in A(B_{n})$ and $y \in [1,\delta]$. Moreover, if $\beta$ is dual positive, then $\SWcr(\beta) \geq 0$ because clockwise rotations never decreases the wall-crossing labelling.
In particular, $\ell_{\sf dual}(\beta) \leq \LWcr(\beta)$ holds.
The assertions for $\SWcr(\beta)$ is proved similarly.
\end{proof}

\section{Comments on Garside normal forms}

We close the paper by discussing an application of the curve diagram method to Garside normal forms. \cite[Section 1]{bgg} contains a concise overview of the normal forms.

For $\beta \in A(B_{n})$, the classical Garside structure introduces the {\em classical normal form}
\[ N_{\sf classical}(\beta) = x_{r}\cdots x_{1}\Delta^{p}\;\;\;(p \in\Z, x_{i}\in [1,\Delta])\]
and the dual Garside structure gives the {\em dual normal form}
\[ N_{\sf dual}(\beta) = y_{s} \cdots y_{1}\delta^{q} \;\;\; (q \in\Z, y_{i}\in [1,\delta]) \]
of $\beta$, respectively. 

The {\em classical supremum} and the {\em classical infimum} of $\beta$ are integers defined by $\sup_{\sf classical}(\beta)= p+r$ and $\inf_{\sf classical}(\beta)= p$, respectively. Similarly, the {\em dual supremum} and the {\em dual infimum} are defined by $\sup_{\sf dual}(\beta)= q+s$ and $\inf_{\sf dual}(\beta)= q$, respectively.  The supremum, infimum and the length are related by the formula
\begin{gather}
\label{eqn:length}
\begin{cases}
\ell_{\sf classical}(\beta) = \max\{0,\sup_{\sf classical}(\beta)\} - \min\{0, \inf_{\sf classical}(\beta)\} \\
\ell_{\sf dual}(\beta) = \max\{0,\sup_{\sf dual}(\beta)\} - \min\{0, \inf_{\sf dual}(\beta)\}. \\
\end{cases}
\end{gather}

By (\ref{eqn:length}), Theorem \ref{theorem:lengthformula_c} and Theorem \ref{theorem:lengthformula_dual} actually prove the following relationships between the supremum/infimum in Garside theory and the labelings of curve diagrams.

\begin{corollary}
Let $\beta \in A(B_{n})$. 
\begin{enumerate}
\item $\LWin(\beta) = \sup_{\sf classical}(\beta)$ and $\SWin(\beta)= \inf_{\sf classical}(\beta)$.
\item $\LWcr(\beta) = \sup_{\sf dual}(\beta)$ and $\SWcr(\beta)= \inf_{\sf dual}(\beta)$.
\end{enumerate}
\end{corollary}

The braid group $\mB_{n}$ also has the classical and the dual Garside structures. 
As a bonus, by comparing the curve diagram theories of $\mB_{2n}$ and $A(B_{n})$, we conclude that the map $\Psi$ preserves both the classical and dual Garside normal forms. 

\begin{corollary}
\label{cor:gmap}
The map $\Psi$ is an embedding that preserves both the classical and the dual Garside normal forms: That is, if the classical and the dual Garside normal form of $\beta \in A(B_{n})$ are
\begin{gather*}
\begin{cases}
N_{\sf classical}(\beta)=x_{r}\cdots x_{1}\Delta^{p}\\
N_{\sf dual}(\beta)=y_{s}\cdots y_{1}\delta^{q},
\end{cases}
\end{gather*}
respectively, then the classical and the dual Garside normal form of the braid $\Psi(\beta) \in \mB_{2n}$ are given by
\begin{gather*}
\begin{cases}
N_{\sf classical}(\Psi(\beta)) = \Psi(x_{r})\cdots\Psi(x_{1})\Psi(\Delta)^{p} \\
N_{\sf dual}(\Psi(\beta)) = \Psi(y_{s})\cdots\Psi(y_{1})\Psi(\delta)^{q},
\end{cases}
\end{gather*}
respectively.

In particular, $\Psi$ is an isometric embedding of $A(B_{n})$ into $\mB_{2n}$ with respect to the word metric on both the classical and the dual simple elements.
\end{corollary}

\begin{proof}
For the braid group $\mB_{2n}$, the curve diagram is defined as an image of the real line segment $[-(n+1),n]$, arranged so that the conditions similar to Definition \ref{defn:curvediagram} (i)--(iii) are satisfied \cite{iw,w}.
Moreover, for the curve diagram of braids, the winding number labeling and the wall-crossing labeling are defined in the similar manner. 

By definition of curve diagram and labelings of braids, the (completed) curve diagrams of elements in $A(B_{n})$ are a special case of the curve diagram of braids. Since the same length formulae of Theorem \ref{theorem:lengthformula_c} and Theorem \ref{theorem:lengthformula_dual} hold for the curve diagrams of braids, we conclude $\Psi$ preserves both the classical and the dual Garside length.
\end{proof}

Corollary \ref{cor:gmap} was already known and has appeared in several places \cite{bdm,dp,pi} by observing that the injection $\Psi$ preserves lattice structures from the classical or the dual simple elements. Here we emphasize that curve diagram argument provides a new geometric proof that avoids the use of Garside theory method.

\end{document}